\documentclass[12pt]{article}%
\usepackage{amsmath}
\usepackage{amsfonts}
\usepackage{amssymb}
\usepackage{graphicx}%
\providecommand{\U}[1]{\protect\rule{.1in}{.1in}}
\newtheorem{theorem}{Theorem}
\newtheorem{acknowledgement}[theorem]{Acknowledgement}

\newtheorem{lemma}[theorem]{Lemma}

\newtheorem{proposition}[theorem]{Proposition}
\newtheorem{remark}[theorem]{Remark}

\newenvironment{proof}[1][Proof]{\noindent\textbf{#1.} }{\ \rule{0.5em}{0.5em}}
\begin{document}

\title{Mehler-Heine type formulas for Charlier and Meixner polynomials.}
\author{Diego Dominici \thanks{e-mail: dominicd@newpaltz.edu}\\Department of Mathematics \\State University of New York at New Paltz \\1 Hawk Dr.\\New Paltz, NY 12561-2443 \\USA}
\maketitle

\begin{abstract}
We derive Mehler--Heine type asymptotic formulas for Charlier and Meixner
polynomials, and also for their associated families. These formulas provide
good approximations for the polynomials in the neighborhood of $x=0,$ and
determine the asymptotic limit of their zeros as the degree $n$ goes to infinity.

\end{abstract}

\strut Keywords: Mehler-Heine formulas, discrete orthogonal polynomials,
associated polynomials, Stieltjes transforms

\strut MSC-class: 41A30 (Primary), 33A65, 33A15, 44A15 (Secondary)

\section{Introduction}

\bigskip\ Mehler--Heine type asymptotic formulas were introduced by Heinrich
Eduard Heine (1821-1881) \cite{MR0204726} and Gustav Ferdinand Mehler
(1835-1895) \cite{MR676041} (see Watson's book \cite[5.71]{MR1349110} for some
historical remarks). They describe the asymptotic behavior of a family of
orthogonal polynomials $P_{n}(x)$ as the degree $n$ tends to infinity, near
one edge of the support of the measure. They have the general form%
\[
\lim_{n\rightarrow\infty}\ f_{n}\left(  x\right)  P_{n}\left(  A+u_{n}%
x\right)  =g(x),\quad x\in S,
\]
where $S$ is a domain in the complex plane, $A$ is a constant, $u_{n}$ is a
given sequence, $g(x)$ is analytic in $S$ and the functions $f_{n}\left(
x\right)  $ are analytic and don't have any zeros in $S,$ for sufficiently
large $n$. The convergence is uniform on compact subsets of $S.$ From
Hurwitz's theorem \cite[4.10e]{MR1008928}, we conclude that for a fixed $k$%
\begin{equation}
x_{n,k}\sim\frac{\zeta_{k}-A}{u_{n}},\quad n\rightarrow\infty\label{zeros}%
\end{equation}
where
\[
x_{n,1}<x_{n,2}<\cdots<x_{n,n}%
\]
are the zeros of $P_{n}(x)$ and $\zeta_{1}<\zeta_{2}<\cdots$ are the zeros of
$g(x).$

Examples of Mehler-Heine formulas include the \emph{Jacobi polynomials}
$P_{n}^{\left(  \alpha,\beta\right)  }(x),$ defined by
\begin{equation}
P_{n}^{\left(  \alpha,\beta\right)  }\left(  x\right)  =\frac{\left(
\alpha+1\right)  _{n}}{n!}\ _{2}F_{1}\left(
\begin{array}
[c]{c}%
-n,n+\alpha+\beta+1\\
\alpha+1
\end{array}
;\frac{1-x}{2}\right)  , \label{Jacobi def}%
\end{equation}
where $\alpha,\beta>-1,$
\[
\ _{p}F_{q}\left(
\begin{array}
[c]{c}%
a_{1},\ldots,a_{p}\\
b_{1},\ldots,b_{q}%
\end{array}
;z\right)  =%
{\displaystyle\sum\limits_{k=0}^{\infty}}
\frac{\left(  a_{1}\right)  _{k}\cdots\left(  a_{p}\right)  _{k}}{\left(
b_{1}\right)  _{k}\cdots\left(  b_{q}\right)  _{k}}\frac{z^{k}}{k!}%
\]
denotes the \emph{Generalized Hypergeometric Function} and $\left(  u\right)
_{k}$ is the \emph{Pochhammer symbol} (or rising factorial) \cite{MR1688958},
\[
\left(  u\right)  _{k}=u\left(  u+1\right)  \cdots\left(  u+k-1\right)  .
\]
For the Jacobi polynomials, we have \cite{MR0372517}%
\begin{equation}
\lim_{n\rightarrow\infty}\ n^{-\alpha}P_{n}^{\left(  \alpha,\beta\right)
}\left(  1-\frac{x^{2}}{2n^{2}}\right)  =\left(  \frac{x}{2}\right)
^{-\alpha}J_{\alpha}\left(  x\right)  , \label{Jacobi MH}%
\end{equation}
where $J_{\alpha}\left(  x\right)  $ is the \emph{Bessel function} of the
first kind \cite{MR1688958}%
\[
J_{\nu}\left(  x\right)  =\left(  \frac{x}{2}\right)  ^{\nu}\frac{1}%
{\Gamma\left(  \nu+1\right)  }\ _{0}F_{1}\left(
\begin{array}
[c]{c}%
-\\
\nu+1
\end{array}
;-\frac{x^{2}}{4}\right)  ,
\]
and $\Gamma\left(  z\right)  $ is the \emph{Gamma function}. The case
$\alpha=\beta=0$ (Legendre polynomials) was the one originally considered by
Mehler and Heine. Extensions of this result for some types of generalized
Jacobi polynomials were studied in \cite{MR3182023}.

The \emph{Laguerre polynomials} $L_{n}^{\left(  \alpha\right)  }\left(
x\right)  ,$ defined by \cite{MR2656096}
\[
L_{n}^{\left(  \alpha\right)  }\left(  x\right)  =\frac{\left(  \alpha
+1\right)  _{n}}{n!}\ _{1}F_{1}\left(
\begin{array}
[c]{c}%
-n\\
\alpha+1
\end{array}
;x\right)  ,\quad\alpha>-1,
\]
satisfy \cite{MR0372517}%
\[
\lim_{n\rightarrow\infty}\ n^{-\alpha}L_{n}^{\left(  \alpha\right)  }\left(
\frac{x^{2}}{4n}\right)  =\left(  \frac{x}{2}\right)  ^{-\alpha}J_{\alpha
}\left(  x\right)  .
\]
Mehler-Heine type formulas for some classes of \emph{multiple} (also called
\emph{polyorthogonal}) \cite{MR1808581} Jacobi and Laguerre polynomials were
considered in \cite{MR2158528}, \cite{MR3051164}, and \cite{MR834171}.

For the \emph{Hermite polynomials} $H_{n}\left(  x\right)  ,$ defined by
\cite{MR2656096}%
\[
H_{n}\left(  x\right)  =\left(  2x\right)  ^{n}\ _{2}F_{0}\left(
\begin{array}
[c]{c}%
-\frac{n}{2},-\frac{n-1}{2}\\
-
\end{array}
;-\frac{1}{x^{2}}\right)  ,
\]
we have two distinct cases:%
\[
\lim_{n\rightarrow\infty}\frac{\left(  -1\right)  ^{n}}{4^{n}n!}\sqrt{n}%
H_{2n}\left(  \frac{x}{2\sqrt{n}}\right)  =\left(  \frac{x}{2}\right)
^{\frac{1}{2}}J_{-\frac{1}{2}}\left(  x\right)  ,
\]
and%
\[
\lim_{n\rightarrow\infty}\frac{\left(  -1\right)  ^{n}}{4^{n}n!}%
H_{2n+1}\left(  \frac{x}{2\sqrt{n}}\right)  =\left(  2x\right)  ^{\frac{1}{2}%
}J_{\frac{1}{2}}\left(  x\right)  .
\]
The Laguerre polynomials and the Hermite polynomials are related by the
quadratic transformations \cite{MR2656096}
\begin{align*}
H_{2n}\left(  x\right)   &  =\left(  -1\right)  ^{n}n!2^{2n}L_{n}^{\left(
-\frac{1}{2}\right)  }\left(  x^{2}\right) \\
H_{2n+1}\left(  x\right)   &  =\left(  -1\right)  ^{n}n!2^{2n+1}%
xL_{n}^{\left(  \frac{1}{2}\right)  }\left(  x^{2}\right)  .
\end{align*}

All of the Mehler-Heine formulas above can be derived from the result
\cite{MR2656096}%
\[
\underset{\lambda\rightarrow\infty}{\lim}\ _{p}F_{q}\left(
\begin{array}
[c]{c}%
a_{1},\ldots,a_{p-1},\lambda a_{p}\\
b_{1},\ldots,b_{q}%
\end{array}
;\frac{x}{\lambda}\right)  =\ _{p-1}F_{q}\left(
\begin{array}
[c]{c}%
a_{1},\ldots,a_{p-1}\\
b_{1},\ldots,b_{q}%
\end{array}
;a_{p}x\right)  .
\]
In \cite{MR1184309}, A. Aptekarev generalized (\ref{Jacobi MH}) in the
following way:

\begin{theorem}
Let $q_{n}(x)$ be an orthonormal system of polynomials defined by%
\[
xq_{n}=b_{n}q_{n+1}+a_{n}q_{n}+b_{n-1}q_{n-1},
\]
with%
\begin{equation}
a_{n}\rightarrow0,\quad b_{n}\rightarrow\frac{1}{2}, \label{Nevai}%
\end{equation}
and suppose that%
\[
\frac{q_{n+1}\left(  1\right)  }{q_{n}\left(  1\right)  }=1+\frac{\alpha
+\frac{1}{2}}{n}+O\left(  n^{-1}\right)  ,\quad\alpha>-1.
\]
Then,%
\[
\lim_{n\rightarrow\infty}\frac{1}{n^{\alpha+\frac{1}{2}}}q_{n}\left(
1-\frac{x^{2}}{2n^{2}}\right)  =x^{-\alpha}J_{\alpha}\left(  x\right)  ,
\]
uniformly on compact subsets of the complex plane.
\end{theorem}

The condition (\ref{Nevai}), indicates that the polynomials $q_{n}\left(
x\right)  $ are orthogonal with respect to a measure $\mu\left(  x\right)  $
that is supported on the interval $\left[  -1,1\right]  $ and belongs to the
\emph{Nevai class} $\mathcal{M}$ \cite{MR519926}. Aptekarev's result was
extended in \cite{MR1835990}. Similar results for multiple orthogonal
polynomials were obtained in \cite{MR2591115} and \cite{MR2541225}

A somehow different type of example is provided by the \emph{Modified Lommel
Polynomials}, defined by \cite{MR0086897}%
\[
h_{n,\nu}(x)=\left(  \nu\right)  _{n}\left(  2x\right)  ^{n}\ _{2}F_{3}\left(
\begin{array}
[c]{c}%
-\frac{n}{2},\frac{1-n}{2}\\
\nu,-n,1-\nu-n
\end{array}
;-\frac{1}{x^{2}}\right)  ,\quad\nu>0.
\]
In this case,%
\[
\lim_{n\rightarrow\infty}\frac{\left(  2x\right)  ^{1-\nu-n}}{\Gamma\left(
n+\nu\right)  }h_{n,\nu}(x)=J_{\nu-1}\left(  x^{-1}\right)  ,\quad x\neq0,
\]
uniformly on compact subsets of $\mathbb{C}\backslash\{0\}.$ If we fix the
value of $x$ (say $x=\frac{1}{2}),$ we obtain a different family of orthogonal
polynomials in the variable $\nu$ \cite{MR0219769}%
\[
R_{n}\left(  z\right)  =h_{n,z}\left(  \frac{1}{2}\right)  .
\]
For these polynomials, we have%
\[
\lim_{n\rightarrow\infty}\frac{1}{\Gamma\left(  n+z\right)  }R_{n}\left(
z\right)  =J_{z-1}\left(  2\right)  .
\]

Generalizations of Mehler-Heine type formulas in the context of Riemannian
geometry were given in \cite{MR2131449}, \cite{MR2599266}, \cite{MR0425521},
and \cite{MR0420158}. Extensions to polynomials in several variables were
studied in \cite{MR1104827}.

Mehler-Heine type formulas have been extensively used in the theory of
\emph{Sobolev orthogonal polynomials} see (among many other articles)
\cite{MR2775140}, \cite{MR2727628}, \cite{MR1946880}, \cite{MR2451185},
\cite{MR2903969}, \cite{MR2639172}, and \cite{MR2134372}.

The work presented in this paper was motivated by a question that Professor
Juan Jos\'{e} Moreno-Balc\'{a}zar asked during our visit to the Universidad de
Almer\'{\i}a in 2010. He was wondering if it would be possible to have
Mehler-Heine type formulas for discrete orthogonal polynomials, since this
would help in calculations involving asymptotics of discrete Sobolev
polynomials \cite{MR2787289}.

The answer is affirmative, and we have obtained results for the Charlier and
Meixner polynomials. These are the only two infinite families of classical
orthogonal polynomials in the discrete lattice \{$0,1,2,\ldots\}.$

\section{Preliminaries}

Let $\psi\left(  t\right)  $ be a bounded, non-decreasing function on
$\mathbb{R},$ with finite \emph{moments}%
\begin{equation}
\mu_{n}=\int\limits_{\mathbb{R}}t^{n}d\psi\left(  t\right)  <\infty,\quad
n=0,1,\ldots\label{moments}%
\end{equation}
and assume that the set
\[
\mathfrak{S}\left(  \psi\right)  =\left\{  t\in\mathbb{R\ }|\ \psi\left(
t+\delta\right)  -\psi\left(  t-\delta\right)  >0\quad\forall\delta>0\right\}
\]
(called the \emph{spectrum} of $\psi)$ is infinite. Under these assumptions,
there exists a unique sequence of monic polynomials $\widehat{P}_{n}(x),$ with
$\deg\widehat{P}_{n}=n,$ such that%
\[
\int\limits_{\mathbb{R}}\widehat{P}_{n}(t)\widehat{P}_{m}\left(  t\right)
d\psi\left(  t\right)  =K_{n}\delta_{n,m},\quad K_{n}>0.
\]
The polynomials $\widehat{P}_{n}(x)$ satisfy a \emph{three-term recurrence
relation}%
\begin{equation}
x\widehat{P}_{n}=\widehat{P}_{n+1}+b_{n}\widehat{P}_{n}+c_{n}\widehat{P}%
_{n-1}, \label{3-term}%
\end{equation}
with initial conditions $\widehat{P}_{-1}\left(  x\right)  =0,\quad\widehat
{P}_{0}\left(  x\right)  =1.$

The inverse problem of finding a distribution function $\psi\left(  t\right)
$ satisfying (\ref{moments}) is called the \emph{Hamburger} \emph{moment
problem }\cite{MR0184042}, \cite{MR0008438}, \cite{MR1627806}. The moment
problem is called \emph{determinate} if there exists a unique solution, and
\emph{indeterminate} otherwise \cite{MR2242752}, \cite{MR1379118}. A possible
criterion for the determinacy of the problem is due to Carleman
\cite{MR986686}. Carleman's Theorem says that the problem is determinate if%
\begin{equation}%
{\displaystyle\sum\limits_{n=1}^{\infty}}
\frac{1}{\sqrt{c_{n}}}=\infty, \label{carleman}%
\end{equation}
where the coefficients $c_{n}>0$ were defined in (\ref{3-term}).

The \emph{associated orthogonal polynomials} $P_{n}^{\ast}\left(  x\right)  $
are defined by \cite{MR2006292}%
\[
xP_{n}^{\ast}=P_{n+1}^{\ast}+b_{n}P_{n}^{\ast}+c_{n}P_{n-1}^{\ast},\quad
P_{0}^{\ast}=0,\quad P_{1}^{\ast}=1.
\]
Note that $\deg P_{n}^{\ast}\left(  x\right)  =n-1.$ We have \cite{MR0481884}%
\[
\mu_{0}P_{n}^{\ast}\left(  x\right)  =\int\limits_{\mathbb{R}}\frac
{\widehat{P}_{n}\left(  x\right)  -\widehat{P}_{n}\left(  t\right)  }%
{x-t}d\psi\left(  t\right)  ,
\]
where $\mu_{0}$ was defined in (\ref{moments}). The associated classical
discrete orthogonal polynomials were studied in \cite{MR1386383},
\cite{MR1807872}, \cite{MR1995218}, \cite{MR968940}, \cite{MR1399900},
\cite{MR1455252}, and \cite{MR1985718}.

The connection between $\widehat{P}_{n}(x),P_{n}^{\ast}\left(  x\right)  $ and
the distribution function $\psi\left(  t\right)  $ is given by \emph{Markov's
theorem}:%
\begin{equation}
\underset{n\rightarrow\infty}{\lim}\mu_{0}\frac{P_{n}^{\ast}\left(  z\right)
}{\widehat{P}_{n}(z)}=\int\limits_{\mathbb{R}}\frac{d\psi\left(  t\right)
}{z-t},\quad z\notin\Lambda, \label{markov}%
\end{equation}
where $\Lambda=\left[  \inf\left(  \mathfrak{S}\right)  ,\sup\left(
\mathfrak{S}\right)  \right]  ,$ and the convergence is uniform in compact
subsets of $\mathbb{C}\backslash\Lambda.$ The original theorem was proved when
$\Lambda$ is a finite interval, but it is also true as long as the
corresponding Hamburger moment problem is determinate \cite{MR1285262},
\cite{MR1136928}.

The function
\[
S(z)=\int\limits_{\mathbb{R}}\frac{d\psi\left(  t\right)  }{z-t},\quad
z\notin\Lambda,
\]
is called the \emph{Stieltjes transform} of $\psi\left(  t\right)  .$ For the
class of discrete distributions that we are considering, we have the following result.

\begin{lemma}
The Stieltjes transform of the distribution
\[
\psi\left(  t\right)  =%
{\displaystyle\sum\limits_{k=0}^{\infty}}
\frac{\left(  a_{1}\right)  _{k}\cdots\left(  a_{p}\right)  _{k}}{\left(
b_{1}\right)  _{k}\cdots\left(  b_{q}\right)  _{k}}\frac{c^{k}}{k!}u\left(
t-k\right)  ,
\]
where $u(t)$ is the unit step function%
\[
u(t)=\left\{
\begin{array}
[c]{c}%
0,\quad t<0\\
1,\quad t\geq0
\end{array}
\right.  ,
\]
is given by%
\begin{equation}
S(z)=\frac{1}{z}\ _{p+1}F_{q+1}\left(
\begin{array}
[c]{c}%
-z,a_{1},\ldots,a_{p}\\
1-z,b_{1},\ldots,b_{q}%
\end{array}
;c\right)  ,\quad z\in\mathbb{C}\backslash\lbrack0,\infty). \label{stieltjes}%
\end{equation}

\end{lemma}

\begin{proof}
Since for all $z\in\mathbb{C}\backslash\lbrack0,\infty)$%
\[
\frac{\left(  -z\right)  _{k}}{\left(  1-z\right)  _{k}}=\prod\limits_{j=0}%
^{k-1}\frac{-z+j}{1-z+j}=\frac{z}{z-k},
\]
we have%
\[
S(z)=%
{\displaystyle\sum\limits_{k=0}^{\infty}}
\frac{\left(  a_{1}\right)  _{k}\cdots\left(  a_{p}\right)  _{k}}{\left(
b_{1}\right)  _{k}\cdots\left(  b_{q}\right)  _{k}}\frac{c^{k}}{k!}\frac
{1}{z-k}=\frac{1}{z}%
{\displaystyle\sum\limits_{k=0}^{\infty}}
\frac{\left(  a_{1}\right)  _{k}\cdots\left(  a_{p}\right)  _{k}}{\left(
b_{1}\right)  _{k}\cdots\left(  b_{q}\right)  _{k}}\frac{c^{k}}{k!}%
\frac{\left(  -z\right)  _{k}}{\left(  1-z\right)  _{k}},
\]
and the result follows.
\end{proof}

\begin{remark}
From the representation%
\[
S(z)=%
{\displaystyle\sum\limits_{k=0}^{\infty}}
\frac{\left(  a_{1}\right)  _{k}\cdots\left(  a_{p}\right)  _{k}}{\left(
b_{1}\right)  _{k}\cdots\left(  b_{q}\right)  _{k}}\frac{c^{k}}{k!}\frac
{1}{z-k},
\]
it is clear that $S(z)$ is a meromorphic function with simple poles at
$z=0,1,\ldots,$ and simple zeros between them. Since the reciprocal Gamma
function is an entire function with simple zeros at $z=0,-1,\ldots,$ we see
that $\frac{S(z)}{\Gamma\left(  -z\right)  }$ is an entire function with
infinitely many simple zeros located in the intervals $\left(  n,n+1\right)
,$ $n=0,1,\ldots.$
\end{remark}

The following result is known as \emph{Tannery's theorem} \cite{tannery}.
Although there are many proofs available in the literature \cite{Bromwich},
\cite{MR1903157}, \cite{MR0079110}, we include one for the sake of completeness.

\begin{theorem}
\label{Th1}\smallskip Suppose that we have%
\[
l_{k}\leq a_{k}\left(  n\right)  \leq u_{k},\quad0\leq k\leq n,
\]
that%
\[
\underset{n\rightarrow\infty}{\lim}a_{k}\left(  n\right)  =A_{k},\quad
k=0,1,\ldots,
\]
and that
\[%
{\displaystyle\sum\limits_{k=0}^{\infty}}
l_{k},\
{\displaystyle\sum\limits_{k=0}^{\infty}}
A_{k},\
{\displaystyle\sum\limits_{k=0}^{\infty}}
u_{k}%
\]
are all convergent series. Then,%
\[
\underset{n\rightarrow\infty}{\lim}%
{\displaystyle\sum\limits_{k=0}^{n}}
a_{k}\left(  n\right)  =%
{\displaystyle\sum\limits_{k=0}^{\infty}}
A_{k}.
\]

\end{theorem}

\begin{proof}
Let $p<n$ be natural numbers and%
\[
x_{n}=%
{\displaystyle\sum\limits_{k=0}^{n}}
a_{k}\left(  n\right)  .
\]
Then,%
\[%
{\displaystyle\sum\limits_{k=0}^{p}}
a_{k}\left(  n\right)  +%
{\displaystyle\sum\limits_{k=p+1}^{n}}
l_{k}\leq x_{n}\leq%
{\displaystyle\sum\limits_{k=0}^{p}}
a_{k}\left(  n\right)  +%
{\displaystyle\sum\limits_{k=p+1}^{n}}
u_{k}.
\]
Letting $n\rightarrow\infty,$ we get%
\[%
{\displaystyle\sum\limits_{k=0}^{p}}
A_{k}+%
{\displaystyle\sum\limits_{k=p+1}^{\infty}}
l_{k}\leq\underset{n\rightarrow\infty}{\underline{\lim}}x_{n}\leq
\underset{n\rightarrow\infty}{\overline{\lim}}x_{n}\leq%
{\displaystyle\sum\limits_{k=0}^{p}}
A_{k}+%
{\displaystyle\sum\limits_{k=p+1}^{\infty}}
u_{k}.
\]
But since $%
{\displaystyle\sum\limits_{k=0}^{\infty}}
l_{k},\
{\displaystyle\sum\limits_{k=0}^{\infty}}
A_{k},\
{\displaystyle\sum\limits_{k=0}^{\infty}}
u_{k}$ converge, we can let $p\rightarrow\infty,$ and obtain%
\[%
{\displaystyle\sum\limits_{k=0}^{\infty}}
A_{k}\leq\underset{n\rightarrow\infty}{\underline{\lim}}x_{n}\leq
\underset{n\rightarrow\infty}{\overline{\lim}}x_{n}\leq%
{\displaystyle\sum\limits_{k=0}^{\infty}}
A_{k}.
\]
Thus,%
\[
\underset{n\rightarrow\infty}{\lim}x_{n}=%
{\displaystyle\sum\limits_{k=0}^{\infty}}
A_{k}.
\]

\end{proof}

\section{Charlier}

The \emph{Charlier polynomials} $C_{n}\left(  x;a\right)  $ are defined by
\cite{MR2656096}
\begin{equation}
C_{n}\left(  x;a\right)  =\ _{2}F_{0}\left(
\begin{array}
[c]{c}%
-n,-x\\
-
\end{array}
;-\frac{1}{a}\right)  , \label{charlier}%
\end{equation}
with $a>0,$ and the corresponding monic polynomials are%
\begin{equation}
\widehat{C}_{n}\left(  x;a\right)  =\left(  -a\right)  ^{n}\ C_{n}\left(
x;a\right)  . \label{monic charlier}%
\end{equation}
The Charlier polynomials are orthogonal with respect to the distribution%
\[
\psi\left(  t\right)  =%
{\displaystyle\sum\limits_{k=0}^{\infty}}
\frac{a^{k}}{k!}u\left(  t-k\right)  ,
\]
and satisfy%
\[%
{\displaystyle\sum\limits_{k=0}^{\infty}}
C_{n}\left(  k;a\right)  C_{m}\left(  k;a\right)  \frac{a^{k}}{k!}=a^{-n}%
e^{a}n!\delta_{n,m},
\]
where $\delta_{n,m}$ is Kronecker's delta%
\[
\delta_{n,m}=\left\{
\begin{array}
[c]{c}%
1,\quad n=m\\
0,\quad n\neq m
\end{array}
\right.  .
\]
Note that if we set $n=m=0$ we get%
\begin{equation}
\mu_{0}=%
{\displaystyle\sum\limits_{k=0}^{\infty}}
\frac{a^{k}}{k!}=e^{a}. \label{moment Charlier}%
\end{equation}

The Mehler-Heine type formula for the Charlier polynomials is the following.

\begin{proposition}
For all complex numbers $x,$we have%
\begin{equation}
\underset{n\rightarrow\infty}{\lim}\frac{a^{n}}{\Gamma\left(  n-x\right)
}C_{n}\left(  x;a\right)  =\frac{e^{a}}{\Gamma\left(  -x\right)  }.
\label{MH charlier}%
\end{equation}

\end{proposition}

\begin{proof}
From the hypergeometric representation (\ref{charlier}), we get
\[
\frac{a^{n}}{\left(  -x\right)  _{n}}C_{n}\left(  x;a\right)  =a^{n}%
{\displaystyle\sum\limits_{j=0}^{n}}
\frac{\left(  -n\right)  _{j}\left(  -x\right)  _{j}}{\left(  1\right)
_{j}\left(  -x\right)  _{n}}\left(  -a\right)  ^{-j}.
\]
Changing the summation variable to $k=n-j,$ we have%
\begin{equation}
\frac{a^{n}}{\left(  -x\right)  _{n}}C_{n}\left(  x;a\right)  =a^{n}%
{\displaystyle\sum\limits_{k=0}^{n}}
\frac{\left(  -n\right)  _{n-k}\left(  -x\right)  _{n-k}}{\left(  1\right)
_{n-k}\left(  -x\right)  _{n}}\left(  -a\right)  ^{k-n}. \label{Charlier0}%
\end{equation}
Using the identity \cite[18:5:10]{MR2466333}
\begin{equation}
\left(  s\right)  _{l}=\left(  s\right)  _{m}\left(  s+m\right)  _{l-m},\quad
m=0,1,\ldots, \label{PochRatio}%
\end{equation}
with $s=-x,$ $l=n$ and $m=n-k,$ we get%
\[
\frac{\left(  -x\right)  _{n-k}}{\left(  -x\right)  _{n}}=\frac{1}{\left(
n-k-x\right)  _{k}}.
\]
From the formula \cite[18:5:1]{MR2466333}
\begin{equation}
\left(  -s\right)  _{l}=\left(  -1\right)  ^{l}\left(  s-l+1\right)  _{l},
\label{PochMinus}%
\end{equation}
with $s=n$ and $l=n-k,$ we have%
\[
\left(  -n\right)  _{n-k}=\left(  -1\right)  ^{n-k}\left(  k+1\right)
_{n-k}.
\]
Thus, we can rewrite (\ref{Charlier0}) as%
\begin{equation}
\frac{a^{n}}{\left(  -x\right)  _{n}}C_{n}\left(  x;a\right)  =%
{\displaystyle\sum\limits_{k=0}^{n}}
\frac{\left(  k+1\right)  _{n-k}}{\left(  1\right)  _{n-k}}\frac{a^{k}%
}{\left(  n-k-x\right)  _{k}}. \label{Charlier1}%
\end{equation}
Using the identity \cite[18:5:1]{MR2466333}%
\begin{equation}
\frac{\left(  s+m\right)  _{l}}{\left(  s\right)  _{l}}=\frac{\left(
s+l\right)  _{m}}{\left(  s\right)  _{m}},\quad m=0,1,\ldots\label{PochRatio1}%
\end{equation}
with $s=1,$ $l=n-k$ and $m=k$ in (\ref{Charlier1}), we obtain%
\begin{equation}
\frac{a^{n}}{\left(  -x\right)  _{n}}C_{n}\left(  x;a\right)  =%
{\displaystyle\sum\limits_{k=0}^{n}}
\frac{\left(  n-k+1\right)  _{k}}{\left(  1\right)  _{k}}\frac{a^{k}}{\left(
n-k-x\right)  _{k}}. \label{Charlier2}%
\end{equation}

But clearly, for all $0\leq k\leq n,$ with $x\leq-1,$
\[
0<\frac{\left(  n-k+1\right)  _{k}}{\left(  n-k-x\right)  _{k}}=%
{\displaystyle\prod\limits_{j=0}^{k-1}}
\frac{n-k+1+j}{n-k-x+j}\leq1,
\]
and for all $k=0,1,\ldots$%
\[
\underset{n\rightarrow\infty}{\lim}\frac{\left(  n-k+1\right)  _{k}}{\left(
n-k-x\right)  _{k}}=1,\quad x\leq-1.
\]
Therefore, from Tannery's theorem we conclude that%
\begin{equation}
\underset{n\rightarrow\infty}{\lim}\frac{a^{n}}{\left(  -x\right)  _{n}}%
C_{n}\left(  x;a\right)  =%
{\displaystyle\sum\limits_{k=0}^{\infty}}
\frac{a^{k}}{k!}=e^{a},\quad x\leq-1. \label{Charlier3}%
\end{equation}
Dividing both sides of (\ref{Charlier3}) by $\Gamma\left(  -x\right)  ,$ we
have%
\begin{equation}
\underset{n\rightarrow\infty}{\lim}\frac{a^{n}}{\Gamma\left(  n-x\right)
}C_{n}\left(  x;a\right)  =\frac{e^{a}}{\Gamma\left(  -x\right)  },\quad
x\leq-1. \label{Charlier4}%
\end{equation}
However, since both sides of the equation are analytic in the whole complex
plane, it follows from the principle of analytic continuation that the formula
is valid for all $x.$
\end{proof}

Other types of asymptotic approximations for $C_{n}\left(  x;a\right)  $ as
$n\rightarrow\infty$ were given in \cite{MR1297273}, \cite{MR2300274},
\cite{MR1857605}, \cite{MR1606887}, and \cite{MR2665812}.

\subsection{Associated polynomials}

The monic Charlier polynomials satisfy the three-term recurrence relation
\cite{MR2656096}%
\[
x\widehat{C}_{n}=\widehat{C}_{n+1}+\left(  n+a\right)  \widehat{C}%
_{n}+an\widehat{C}_{n-1},
\]
with initial conditions
\[
\widehat{C}_{-1}(x;a)=0,\quad\widehat{C}_{0}(x;a)=1.
\]
The associated polynomials $C_{n}^{\ast}\left(  x;a\right)  $ satisfy the same
recurrence, but the initial conditions are
\[
C_{0}^{\ast}(x;a)=0,\quad C_{1}^{\ast}(x;a)=1.
\]
Using Carleman's Theorem (\ref{carleman}), we see that the moment problem is
determinate. Hence, from (\ref{markov}) and (\ref{stieltjes}) we have%
\begin{equation}
\underset{n\rightarrow\infty}{\lim}e^{a}\frac{C_{n}^{\ast}\left(  z;a\right)
}{\widehat{C}_{n}\left(  z;a\right)  }=\frac{1}{z}\ _{1}F_{1}\left(
\begin{array}
[c]{c}%
-z\\
1-z
\end{array}
;a\right)  ,\quad z\in\mathbb{C}\backslash\lbrack0,\infty)
\label{markov charlier}%
\end{equation}
where we have used (\ref{moment Charlier}).

From (\ref{monic charlier}) and (\ref{MH charlier}), it follows that%
\begin{equation}
\underset{n\rightarrow\infty}{\lim}\frac{\left(  -1\right)  ^{n}}%
{\Gamma\left(  n-x\right)  }\widehat{C}_{n}\left(  x;a\right)  =\frac{e^{a}%
}{\Gamma\left(  -x\right)  }. \label{MH charlier 1}%
\end{equation}
Thus, from (\ref{markov charlier}) and (\ref{MH charlier 1}) we obtain%
\begin{equation}
\underset{n\rightarrow\infty}{\lim}\frac{\left(  -1\right)  ^{n}}%
{\Gamma\left(  n-x\right)  }C_{n}^{\ast}\left(  x;a\right)  =\frac{1}%
{x\Gamma\left(  -x\right)  }\ _{1}F_{1}\left(
\begin{array}
[c]{c}%
-x\\
1-x
\end{array}
;a\right)  . \label{MH charlier 2}%
\end{equation}

Using the formulas \cite[13.6.5, 8.2.6]{MR2723248}
\[
_{1}F_{1}\left(
\begin{array}
[c]{c}%
b\\
b+1
\end{array}
;-z\right)  =bz^{-b}\gamma\left(  b,z\right)  =b\Gamma\left(  b\right)
\gamma^{\ast}\left(  b,z\right)  ,
\]
we get%
\begin{equation}
\underset{n\rightarrow\infty}{\lim}\frac{\left(  -1\right)  ^{n}}%
{\Gamma\left(  n-x\right)  }C_{n}^{\ast}\left(  x;a\right)  =-\gamma^{\ast
}\left(  -x,-a\right)  , \label{MH charlier 3}%
\end{equation}
where $\gamma^{\ast}\left(  b,z\right)  $ is the \emph{entire incomplete gamma
function }defined by \cite[8.2.7]{MR2723248}
\[
\gamma^{\ast}\left(  b,z\right)  =\frac{1}{\Gamma\left(  b\right)  }%
{\displaystyle\int\limits_{0}^{1}}
t^{b-1}e^{-zt}dt,
\]
for $\operatorname{Re}\left(  b\right)  >0,$ and by analytic continuation
elsewhere. The function $\gamma^{\ast}\left(  b,z\right)  $ is entire in $b$
and $z,$ and has two zeros in each of the intervals $\left(  2n-2,2n\right)  $
for all $n=1,2,\ldots$ \cite{MR0326994}. It follows from (\ref{zeros}) that
the zeros of $C_{n}^{\ast}\left(  x;a\right)  $ approach the zeros of the
function $\gamma^{\ast}\left(  -x,-a\right)  $ as $n\rightarrow\infty.$

Using the formula \cite[45:6:4]{MR2466333}%
\[
\gamma^{\ast}\left(  b,z\right)  =e^{-z}\sum_{j=0}^{\infty}\frac{z^{j}}%
{\Gamma\left(  b+1+j\right)  },
\]
we can rewrite (\ref{MH charlier 2}) as%
\[
\underset{n\rightarrow\infty}{\lim}\frac{\left(  -1\right)  ^{n}}%
{\Gamma\left(  n-x\right)  }C_{n}^{\ast}\left(  x;a\right)  =-e^{a}%
{\displaystyle\sum\limits_{k=0}^{\infty}}
\frac{\left(  -a\right)  ^{k}}{\Gamma\left(  1-x+k\right)  }.
\]
In Figure 1, we plot the functions%
\[
\frac{1}{\Gamma\left(  28-x\right)  }C_{28}^{\ast}\left(  x;1.23\right)
,\quad-\gamma^{\ast}\left(  -x,-1.23\right)  ,
\]
to illustrate the accuracy of (\ref{MH charlier 3}).

\begin{figure}[ptb]
\begin{center}
{\resizebox{15cm}{!}{\includegraphics{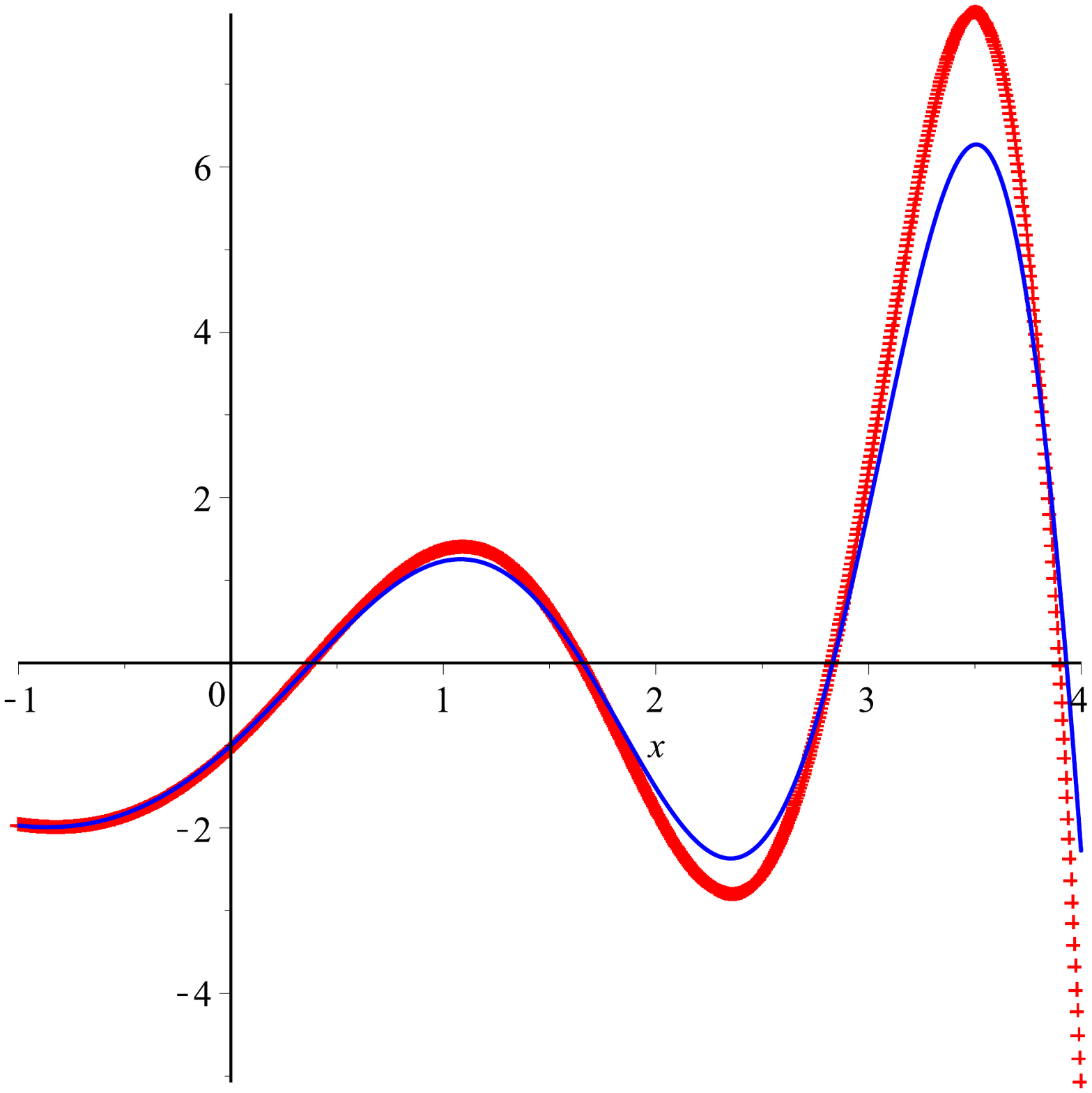}}}
\end{center}
\caption{A plot of the scaled polynomial $C_{28}^{\ast}$ (+++) and the
limiting function (solid line).}%
\end{figure}

\subsection{Meixner}

The \emph{Meixner polynomials} $M_{n}\left(  x;\beta,c\right)  $ are defined
by \cite{MR2656096}%
\begin{equation}
M_{n}\left(  x;\beta,c\right)  =\ _{2}F_{1}\left(
\begin{array}
[c]{c}%
-n,-x\\
\beta
\end{array}
;1-\frac{1}{c}\right)  , \label{Meixner}%
\end{equation}
where $\beta>0$ and $0<c<1.$ The monic polynomials are%
\begin{equation}
\widehat{M}_{n}\left(  x;\beta,c\right)  =\left(  \beta\right)  _{n}\left(
\frac{c}{c-1}\right)  ^{n}\ M_{n}\left(  x;\beta,c\right)  .
\label{monic meixner}%
\end{equation}
The Meixner polynomials are orthogonal with respect to the distribution%
\[
\psi\left(  t\right)  =%
{\displaystyle\sum\limits_{k=0}^{\infty}}
\left(  \beta\right)  _{k}\frac{c^{k}}{k!}u\left(  t-k\right)  ,
\]
and satisfy \cite{MR2656096}%
\[%
{\displaystyle\sum\limits_{k=0}^{\infty}}
M_{n}\left(  x;\beta,c\right)  M_{m}\left(  x;\beta,c\right)  \left(
\beta\right)  _{k}\frac{c^{k}}{k!}=\frac{c^{-n}n!}{\left(  \beta\right)
_{n}\left(  1-c\right)  ^{\beta}}\delta_{n,m}.
\]
In particular, for $n=m=0,$ we get%
\begin{equation}
\mu_{0}=%
{\displaystyle\sum\limits_{k=0}^{\infty}}
\left(  \beta\right)  _{k}\frac{c^{k}}{k!}=\left(  1-c\right)  ^{-\beta}.
\label{moment meixner}%
\end{equation}

The Mehler-Heine type formula for the Meixner polynomials is the following.

\begin{proposition}
For all complex numbers $x,$ we have%
\begin{equation}
\underset{n\rightarrow\infty}{\lim}\ \frac{c^{n}\ \left(  \beta\right)  _{n}%
}{\Gamma\left(  n-x\right)  }M_{n}\left(  x;\beta,c\right)  =\frac{1}{\left(
1-c\right)  ^{\beta+x}\Gamma\left(  -x\right)  }. \label{MH meixner}%
\end{equation}

\end{proposition}

\begin{proof}
From (\ref{Meixner}) and the formula \cite[2.3.14]{MR1688958}
\[
\ _{2}F_{1}\left(
\begin{array}
[c]{c}%
-n,a\\
b
\end{array}
;x\right)  =\frac{\left(  b-a\right)  _{n}}{\left(  b\right)  _{n}}\ _{2}%
F_{1}\left(
\begin{array}
[c]{c}%
-n,a\\
a+1-n-b
\end{array}
;1-x\right)  ,
\]
we get%
\[
M_{n}\left(  x;\beta,c\right)  =\frac{\left(  x+\beta\right)  _{n}}{\left(
\beta\right)  _{n}}\ _{2}F_{1}\left(
\begin{array}
[c]{c}%
-n,-x\\
-x+1-n-\beta
\end{array}
;\frac{1}{c}\right)  .
\]
Thus,%
\[
\left(  \beta\right)  _{n}\ c^{n}\ M_{n}\left(  x;\beta,c\right)  =\left(
x+\beta\right)  _{n}\ c^{n}\ _{2}F_{1}\left(
\begin{array}
[c]{c}%
-n,-x\\
-x+1-n-\beta
\end{array}
;\frac{1}{c}\right)  .
\]
It follows that%
\[
\left(  \beta\right)  _{n}\ \left(  -x\right)  _{n}\ c^{n}\ M_{n}\left(
x;\beta,c\right)  =%
{\displaystyle\sum\limits_{k=0}^{n}}
\frac{\left(  -n\right)  _{k}\left(  -x\right)  _{k}\left(  x+\beta\right)
_{n}\ \left(  -x\right)  _{n}}{\left(  -x+1-n-\beta\right)  _{k}}\frac
{c^{n-k}}{k!},
\]
or%
\begin{equation}
\left(  \beta\right)  _{n}\ c^{n}\ M_{n}\left(  x;\beta,c\right)  =%
{\displaystyle\sum\limits_{k=0}^{n}}
\frac{\left(  -n\right)  _{n-k}\left(  -x\right)  _{n-k}\left(  x+\beta
\right)  _{n}\ }{\left(  -x+1-n-\beta\right)  _{n-k}}\frac{c^{k}}{\left(
n-k\right)  !}. \label{Meixner2}%
\end{equation}
Using (\ref{PochMinus}) and (\ref{PochRatio}) we have%
\[
\frac{\left(  -n\right)  _{n-k}}{\left(  -x+1-n-\beta\right)  _{n-k}}%
=\frac{\left(  k+1\right)  _{n-k}}{\left(  x+\beta+k\right)  _{n-k}}%
=\frac{\left(  1\right)  _{n}}{\left(  1\right)  _{k}}\frac{\left(
x+\beta\right)  _{k}}{\left(  x+\beta\right)  _{n}}.
\]
Hence, we can write (\ref{Meixner2}) in the form%
\begin{equation}
\left(  \beta\right)  _{n}\ c^{n}\ M_{n}\left(  x;\beta,c\right)  =%
{\displaystyle\sum\limits_{k=0}^{n}}
\frac{\left(  1\right)  _{n}}{\left(  1\right)  _{k}}\left(  x+\beta\right)
_{k}\left(  -x\right)  _{n-k}\frac{c^{k}}{\left(  n-k\right)  !},
\label{Meixner3}%
\end{equation}
and therefore%
\begin{equation}
\frac{\left(  \beta\right)  _{n}\ c^{n}\ \ M_{n}\left(  x;\beta,c\right)
}{\left(  -x\right)  _{n}}=%
{\displaystyle\sum\limits_{k=0}^{n}}
\frac{\left(  -x\right)  _{n-k}\ }{\left(  -x\right)  _{n}}\frac{n!}{\left(
n-k\right)  !}\left(  x+\beta\right)  _{k}\frac{c^{k}}{k!}. \label{Meixner4}%
\end{equation}

But since%
\[
\frac{\left(  -x\right)  _{n-k}\ }{\left(  -x\right)  _{n}}\frac{n!}{\left(
n-k\right)  !}=%
{\displaystyle\prod\limits_{j=0}^{k-1}}
\frac{n-j}{n-j-\left(  x+1\right)  }\leq1,\quad x\leq-1,
\]
and
\[
\underset{n\rightarrow\infty}{\lim}\frac{\left(  -x\right)  _{n-k}\ }{\left(
-x\right)  _{n}}\frac{n!}{\left(  n-k\right)  !}=1,\quad x\leq-1,
\]
we can use Tannery's theorem and conclude that%
\begin{equation}
\underset{n\rightarrow\infty}{\lim}\frac{\left(  \beta\right)  _{n}%
\ c^{n}\ \ M_{n}\left(  x;\beta,c\right)  }{\left(  -x\right)  _{n}}=%
{\displaystyle\sum\limits_{k=0}^{\infty}}
\left(  x+\beta\right)  _{k}\frac{c^{k}}{k!}=\left(  1-c\right)  ^{-x-\beta
},\quad x\leq-1. \label{Meixner5}%
\end{equation}
Dividing both sides of (\ref{Meixner5}) by $\Gamma\left(  -x\right)  ,$ we
have%
\[
\underset{n\rightarrow\infty}{\lim}\ \frac{c^{n}\ \left(  \beta\right)  _{n}%
}{\Gamma\left(  n-x\right)  }M_{n}\left(  x;\beta,c\right)  =\frac{1}{\left(
1-c\right)  ^{\beta+x}\Gamma\left(  -x\right)  },\quad x\leq-1.
\]
Since both sides of the equation are analytic in the whole complex plane, it
follows from the principle of analytic continuation that the formula is valid
for all $x.$
\end{proof}

Other asymptotic approximations for $M_{n}\left(  x;\beta,c\right)  $ as
$n\rightarrow\infty$ were studied in \cite{MR3184968}, \cite{MR1486393},
\cite{MR778685}, \cite{MR1217643}, \cite{MR2906191}, and \cite{MR1805998}.

\subsection{Associated polynomials}

The monic Meixner polynomials satisfy the three-term recurrence relation
\cite{MR2656096}%
\[
x\widehat{M}_{n}=\widehat{M}_{n+1}+\frac{n+\left(  n+\beta\right)  c}%
{1-c}\ \widehat{M}_{n}+\frac{n\left(  n+\beta-1\right)  c}{\left(  1-c\right)
^{2}}\ \widehat{M}_{n-1},
\]
with initial conditions
\[
\widehat{M}_{-1}\left(  x;\beta,c\right)  =0,\quad\widehat{M}_{0}\left(
x;\beta,c\right)  =1.
\]
The associated polynomials $M_{n}^{\ast}\left(  x;\beta,c\right)  $ satisfy
the same recurrence, but the initial conditions are
\[
M_{0}^{\ast}\left(  x;\beta,c\right)  =0,\quad M_{1}^{\ast}\left(
x;\beta,c\right)  =1.
\]
Using Carleman's Theorem (\ref{carleman}), we see that the moment problem is
determinate. Hence, from (\ref{markov}) and (\ref{stieltjes}) we have%
\begin{equation}
\underset{n\rightarrow\infty}{\lim}\left(  1-c\right)  ^{-\beta}\frac
{M_{n}^{\ast}\left(  z;\beta,c\right)  }{\widehat{M}_{n}\left(  z;\beta
,c\right)  }=\frac{1}{z}\ _{2}F_{1}\left(
\begin{array}
[c]{c}%
-z,\beta\\
1-z
\end{array}
;c\right)  ,\quad z\in\mathbb{C}\backslash\lbrack0,\infty),
\label{markov meixner}%
\end{equation}
where we have used (\ref{moment meixner}).

From (\ref{monic meixner}) and (\ref{MH meixner}), it follows that%
\begin{equation}
\underset{n\rightarrow\infty}{\lim}\ \frac{\left(  c-1\right)  ^{n}\ }%
{\Gamma\left(  n-x\right)  }\widehat{M}_{n}\left(  x;\beta,c\right)  =\frac
{1}{\left(  1-c\right)  ^{\beta+x}\Gamma\left(  -x\right)  }.
\label{MH meixner 1}%
\end{equation}
Thus, from (\ref{markov meixner}) and (\ref{MH meixner 1}) we obtain%
\begin{equation}
\underset{n\rightarrow\infty}{\lim}\frac{\left(  c-1\right)  ^{n}}%
{\Gamma\left(  n-x\right)  }M_{n}^{\ast}\left(  x;\beta,c\right)
=\frac{\left(  1-c\right)  ^{-x}}{x\Gamma\left(  -x\right)  }\ _{2}%
F_{1}\left(
\begin{array}
[c]{c}%
-x,\beta\\
1-x
\end{array}
;c\right)  . \label{MH meixner 2}%
\end{equation}

Using the formula \cite[8.17.7]{MR2723248}
\[
_{2}F_{1}\left(
\begin{array}
[c]{c}%
a,1-b\\
a+1
\end{array}
;z\right)  =az^{-a}B_{z}\left(  a,b\right)  ,
\]
we get%
\[
\underset{n\rightarrow\infty}{\lim}\frac{\left(  c-1\right)  ^{n}}%
{\Gamma\left(  n-x\right)  }M_{n}^{\ast}\left(  x;\beta,c\right)  =-\left(
\frac{c}{1-c}\right)  ^{x}\frac{1}{\Gamma\left(  -x\right)  }B_{c}\left(
-x,1-\beta\right)  ,
\]
where $B_{z}\left(  a,b\right)  $ is the \emph{incomplete Beta function
}defined by \cite[58:3:5]{MR2466333}
\[
B_{z}\left(  a,b\right)  =z^{a}%
{\displaystyle\int\limits_{0}^{1}}
t^{a-1}\left(  1-zt\right)  ^{b-1}dt,
\]
for $a,b>0,$ $z\in\left[  0,1\right]  ,$ and by analytic continuation
elsewhere. It follows from (\ref{zeros}) that the zeros of $M_{n}^{\ast
}\left(  x;\beta,c\right)  $ approach the zeros of the function $B_{c}\left(
-x,1-\beta\right)  $ as $n\rightarrow\infty.$

In Figure 2, we plot the functions%
\[
\frac{\left(  c-1\right)  ^{28}}{\Gamma\left(  28-x\right)  }M_{28}^{\ast
}\left(  x;1.23,0.36\right)  ,\quad\frac{\left(  1-0.36\right)  ^{-x}}%
{x\Gamma\left(  -x\right)  }\ _{2}F_{1}\left(
\begin{array}
[c]{c}%
-x,1.23\\
1-x
\end{array}
;0.36\right)  ,
\]
to illustrate the accuracy of (\ref{MH meixner 2}).

\begin{figure}[ptb]
\begin{center}
{\resizebox{15cm}{!}{\includegraphics{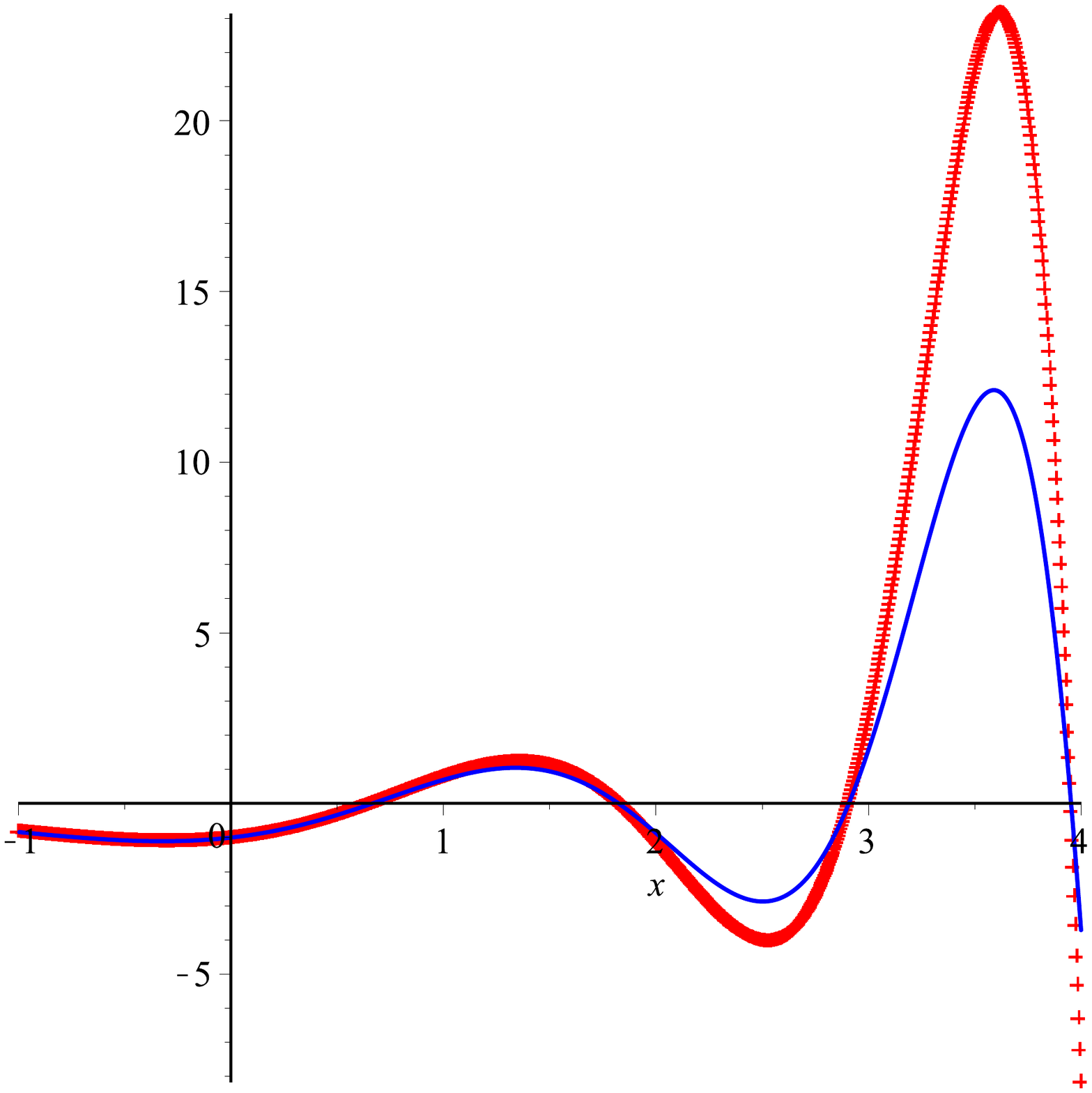}}}
\end{center}
\caption{A plot of the scaled polynomial $M_{28}^{\ast}$ (+++) and the
limiting function (solid line).}%
\end{figure}

\section{Conclusion}

We have derived Mehler-Heine type formulas for the Charlier and Meixner
families and their associated polynomials. We plan to extend this
investigation to include other discrete orthogonal polynomials of class one
\cite{discrete}.

\begin{acknowledgement}
We thank Professor Juan Jos\'{e} Moreno-Balc\'{a}zar for suggesting the
problem. This work was completed while visiting the Research Institute for
Symbolic Computation (RISC). We wish to thank Professor Peter Paule for being
the most gracious host during our stay.
\end{acknowledgement}

\end{document}